\documentclass[12pt]{amsart}
\usepackage[T1]{fontenc}
\usepackage[utf8]{inputenc}
\usepackage{lmodern}
\usepackage{amsmath,amssymb,amsthm,mathtools}
\usepackage{microtype,needspace}
\usepackage[noadjust]{cite}
\usepackage[top=28mm,bottom=27mm,left=27mm,right=27mm]{geometry}
\usepackage[hidelinks]{hyperref}
\allowdisplaybreaks[2]

\numberwithin{equation}{section}
\newtheorem{theorem}{Theorem}[section]
\newtheorem{proposition}[theorem]{Proposition}
\newtheorem{lemma}[theorem]{Lemma}
\newtheorem{corollary}[theorem]{Corollary}
\theoremstyle{remark}
\newtheorem{remark}[theorem]{Remark}
\DeclareMathOperator{\trace}{trace}
\DeclareMathOperator{\diver}{div}
\DeclareMathOperator{\Scal}{Scal}
\newcommand{\R}{\mathbb R}

\newcommand{\SSS}{\mathbb S}
\newcommand{\HHvec}{\mathbf H}

\newcommand{\ip}[2]{\langle #1,#2\rangle}

\allowdisplaybreaks[3]

\hypersetup{pdftitle={Biharmonic hypersurfaces in space forms},pdfauthor={Yu Fu and Min-Chun Hong and Dan Yang}}

\title[Biharmonic hypersurfaces]{Biharmonic hypersurfaces in space forms}
\author{Yu Fu}
\address{Center of Applied Mathematics, Dongbei University of Finance and Economics, Dalian 116025, P. R. China}
\email{yufu@dufe.edu.cn}
\author{Min-Chun Hong}
\address{Department of Mathematics, The University of Queensland, Brisbane, QLD 4072, Australia}
\email{hong@maths.uq.edu.au}
\author{Dan Yang}
\address{School of Mathematics, Liaoning University, Shenyang 110036, P. R. China}
\email{dlutyangdan@126.com}
\date{}

\begin{document}
\begin{abstract}
We prove that every biharmonic hypersurface in a space form of nonpositive sectional curvature is minimal in arbitrary dimension. This settles the hypersurface case of the generalized Chen's conjecture in hyperbolic space and gives a unified treatment of nonpositive space forms. For biharmonic hypersurfaces in the unit sphere, we derive quantitative restrictions on every possible nonconstant-mean-curvature solution. In particular, we obtain pointwise criteria forcing constant mean curvature, a strict scalar-curvature bound in dimensions at least five, and local classification results above the classical CMC gap threshold. These results provide further evidence for the BMO conjecture.
\end{abstract}
\maketitle

\section{Introduction}

Biharmonic maps generalize classical harmonic maps as critical points of the bienergy functional. Following Eells--Sampson~\cite{EellsSampson1966} and Eells--Lemaire~\cite{EellsLemaire1983}, the bienergy of a smooth map $\phi:(M^n,g)\to(N^k,\widetilde g)$ is defined by
\begin{align*}
 E_2(\phi)=\frac12\int_{M^n}|\tau(\phi)|^2\,dv_g,
\end{align*}
where $\tau(\phi)=\trace\nabla d\phi$ is the tension field. A map is biharmonic if it is a critical point of $E_2$ with respect to compactly supported variations. Jiang~\cite{Jiang1986} derived the corresponding Euler--Lagrange equation
\begin{align}
 \tau_2(\phi)=-\Delta\tau(\phi)
 -\trace R^N(d\phi,\tau(\phi))d\phi=0,\label{bitensioneq}
\end{align}
where $\Delta$ is the rough Laplacian on sections of the pullback bundle \(\varphi^{-1}TN\), and $R^N$ is the curvature tensor of $N^k$. An isometric immersion  $\phi:M^n\to N^k$ is called biharmonic if it is biharmonic as a map. Let $\HHvec$ denote the mean curvature vector. Since $\tau(\phi)=n\HHvec$, every minimal immersion is biharmonic. A biharmonic immersion is called {\em proper biharmonic} if it is not minimal.

For an isometric immersion into Euclidean space, the biharmonic equation \eqref{bitensioneq} reduces to $\Delta\HHvec=0$. Chen~\cite{Chen1991} and Jiang~\cite{Jiang1987} independently proved that every biharmonic surface in $\R^3$ is minimal. In 1991, Chen~\cite{Chen1991} subsequently formulated the following longstanding conjecture.

\medskip
\noindent\textbf{Chen's conjecture:} \textit{Every biharmonic submanifold in a Euclidean space is minimal.}
\medskip

Chen's  conjecture on   hypersurfaces   was proved by Hasanis--Vlachos~\cite{HasanisVlachos1995} in $\R^4$ and by Fu--Hong--Zhan~\cite{FHZ2021,FHZ2023} in $\R^5$  and $\R^6$  (see also~\cite{Chen2014,OuChen2020}).
In arbitrary dimension, some results under restrictions such as the number of distinct principal curvatures, convexity, properness, cohomogeneity, or holonomicity appear in~\cite{Dimitric1992,Fu2015,ChenMunteanu2013,Luo2014,AkutagawaMaeta2013,MontaldoOniciucRatto2016, KoisoUrakawa2018,BibiSoretVille2025}. Very recently, Chen's conjecture on  hypersurfaces in $\R^{n+1}$ was settled by Abraham in \cite{Abraham2026}.

In biharmonic map theory, the nonexistence and classification of biharmonic submanifolds in nonflat ambient spaces are of fundamental importance. Caddeo, Montaldo and Oniciuc \cite{CMO2001} in 2001 proposed the generalized Chen's conjecture: {\em{every biharmonic submanifold in a Riemannian manifold of nonpositive sectional curvature is minimal}}. Although counterexamples were constructed in negatively curved manifolds with variable sectional curvature in \cite{OuTang2012}, the generalized Chen's conjecture remains open in hyperbolic spaces. The study of biharmonic hypersurfaces in hyperbolic spaces is   an important problem in the theory of biharmonic submanifolds. The hypersurface case of the generalized Chen's conjecture in dimensions three, four, and five was established in~\cite{BMO2010,GuanLiVrancken2021,FHZ2023}, respectively.

\medskip
\noindent\textbf{Generalized Chen's conjecture in hyperbolic space:} \textit{Every biharmonic submanifold in a hyperbolic space is minimal.}
\medskip

 Let $N^{n+1}(c)$ be a space form of constant sectional curvature $c$ with $n\geq4$. Throughout the paper, hypersurfaces are assumed to be connected unless otherwise stated. The nonpositive curvature case will be treated for $c\le0$. In the positive curvature case, after the usual homothetic normalization, we take $c=1$ and write the ambient space as the unit sphere $\SSS^{n+1}$. We denote by $H$ the scalar mean curvature relative to a local unit normal and by $S=\trace A^2$ the squared norm of the second fundamental form.

 Our first result establishes the nonexistence of proper biharmonic hypersurfaces in nonpositive space forms in arbitrary dimension.

 \begin{theorem}\label{thm:nonpositive}
 	Every biharmonic hypersurface in a space form $N^{n+1}(c)$ with $c\le0$ is minimal.
 \end{theorem}

 For $c<0$, this settles the hypersurface case of the generalized Chen's conjecture in hyperbolic space. When $c=0$, it provides an independent proof of the Euclidean case, which was recently established in~\cite{Abraham2026}. Our argument builds on the principal-curvature reduction developed in~\cite{FHZ2021,FHZ2023} and combines it with a continuation analysis on the maximal interval of existence of the resulting differential system. Rather than carrying out a dimension-dependent algebraic elimination, we regard the reduced principal-curvature equations as an autonomous polynomial dynamical system and exploit their maximal analytic continuation. A key feature of this argument is that the continuation step is independent of the dimension.

In contrast, the unit sphere admits many proper biharmonic hypersurfaces. The standard examples are the small hypersphere $\SSS^n(1/\sqrt2)$ and the products $\SSS^k(1/\sqrt2)\times\SSS^{n-k}(1/\sqrt2)$ with $1\le k\le n-1$ and $k\ne n/2$, see~\cite{CMO2001,CMO2002,BMO2008}. These examples have constant mean curvature and motivated Balmu\c{s}, Montaldo and Oniciuc~\cite{BMO2008} to formulate the following conjecture.

\medskip
\noindent\textbf{The BMO conjecture:} \textit{Every proper biharmonic hypersurface in $\SSS^{n+1}$ has constant mean curvature.}
\medskip

The BMO conjecture is known in several low-dimensional or geometrically restricted cases, see \cite{CMO2001,BMO2010,GuanLiVrancken2021,BMO2008,FuHong2018,LuoMaeta2017,MaetaOu2020,FetcuLoubeauOniciuc2021} for examples. The general arbitrary-dimensional case remains open. The estimates obtained below give pointwise restrictions on a possible non-CMC solution. Unique continuation for biharmonic maps was proved in~\cite{BrandingOniciuc2019}, and unique continuation for constant mean curvature (CMC) under an additional gradient inequality was studied in~\cite{BibiLoubeauOniciuc2021}. We refer the reader to the recent survey~\cite{FetcuOniciuc2022} for a comprehensive overview of related results.

For biharmonic hypersurfaces in the unit sphere $\SSS^{n+1}$, we will prove that every non-CMC solution satisfies
\[
 H^2<\frac{4(n-1)}{n(n+8)}
\]
on $M^n$. No completeness, compactness, constant scalar curvature, or a priori CMC assumption is imposed. The strict bound for $H^2$ gives the following pointwise criterion.

\begin{theorem}\label{cor:one-point}
Let $M^n$ be a biharmonic hypersurface in $\SSS^{n+1}$. If $H^2(p)\ge4(n-1)/[n(n+8)]$ at some point $p\in M^n$, then $M^n$ has constant mean curvature.
\end{theorem}

Bibi, Loubeau and Oniciuc~\cite[Theorem~3.10]{BibiLoubeauOniciuc2021} proved that a proper biharmonic hypersurface satisfying $H^2>4(n-1)/[n(n+8)]$ everywhere has constant mean curvature, provided that $H$ is constant on a nonempty open subset. Theorem~\ref{cor:one-point} replaces these hypotheses by a condition at a single point and includes equality.

We employ the Gauss equation to give a criterion in terms of the scalar curvature.

\begin{corollary}\label{cor:scalar-point}
Let $M^n$ be a biharmonic hypersurface in $\SSS^{n+1}$ {\rm ($n\ge5$)}. If $\Scal(p)\ge n(n^2+10n-20)/(n+8)$ at one point, then $M^n$ has constant mean curvature.
\end{corollary}

Combining Theorem~\ref{cor:one-point} with the constant mean curvature gap in~\cite{BO2012}, we obtain the following classification.

\begin{theorem}\label{thm:upper-classification}
Let $M^n$ be a biharmonic hypersurface in $\SSS^{n+1}$ {\rm ($n\ge4$)}. If $|H(p)|>(n-2)/n$ at one point $p$, then $M^n$ is locally congruent to an open part of $\SSS^n(1/\sqrt2)$. If $|H(p)|=(n-2)/n$ at one point $p$, then $M^n$ is locally congruent to an open part of
\begin{equation}\label{eq:Clifford-one}
 \SSS^1\!\left(\frac1{\sqrt2}\right)
 \times\SSS^{n-1}\!\left(\frac1{\sqrt2}\right)
 \subset\SSS^{n+1}.
\end{equation}
\end{theorem}

The gap value $(n-2)/n$ in Theorem~\ref{thm:upper-classification} is the one obtained in~\cite{BO2012} for proper biharmonic CMC hypersurfaces. The point here is that constant mean curvature is not assumed: the one-point hypothesis first forces the CMC property by Theorem~\ref{cor:one-point}, after which the gap theorem applies. Thus Theorem~\ref{thm:upper-classification} removes the a priori CMC assumption in this range, while retaining the sharp endpoint and its equality case from~\cite{BO2012}.

We briefly describe the argument. The starting point is the overdetermined differential system~\eqref{eq:system}, obtained from the Gauss and Codazzi equations and developed in~\cite{FuHong2018,FHZ2021,FHZ2023}. The reduction uses recurrence relations for the power sums and successive elimination of the connection coefficients to obtain polynomial restrictions on the distinguished principal curvature. We retain this framework and the notation $F_k,H_k,G_{l,m}$. The resulting differential system is then studied on its maximal interval of existence, avoiding a separate dimension-dependent elimination.

Write $U=\{p\in M:\nabla H(p)\ne0\}$. On $U$, the tangential biharmonic equation identifies $\nabla H$ as a principal direction with principal curvature $\lambda=-nH/2$. Working locally where the principal curvature multiplicities are constant, one obtains system~\eqref{eq:system},
whose trace relation eliminates $\lambda$ from the first-order equations, leaving a polynomial autonomous system for the $2n-2$ functions $\lambda_i,a_i$. The scalar biharmonic equation is an additional condition, expressed by the polynomial $\mathcal P$ in~\eqref{eq:P-definition}. Since the right-hand side of the autonomous system is polynomial, its solutions are real analytic on every maximal interval of existence. Hence the scalar constraint extends from the initial open interval to the whole maximal interval by analyticity. Thus the question raised in~\cite[Remark~3.4]{FHZ2023} becomes a nonexistence problem for solutions with nonconstant $\lambda$.

The continuation argument is applied directly to the curvature-dependent principal-curvature system. The key estimate is Lemma~\ref{lem:bounded-factors}. Positive solutions of $\rho_i'=-a_i\rho_i$ give the conserved quantities~\eqref{eq:pair-invariant} and the differentiated products~\eqref{eq:product-monotonicity}. The trace relation provides an index for which the corresponding product is bounded along a sequence approaching the endpoint, and monotonicity extends this bound to the whole forward interval. Logarithmic estimates then control all but one pair of variables; the trace relation and one further integration control $\lambda$ and the remaining pair. These bounds allow us to apply the ODE continuation theorem at a finite endpoint. The normalization also allows the conserved quantity in~\eqref{eq:pair-invariant} to be nonpositive when $c<0$.

For $c=0$ or $c<0$, the scalar equation forces positive growth after a choice of signs and parameter direction. The continuation estimate excludes a finite endpoint, and a shifted quotient involving $\lambda-\sqrt{-c}$ reduces a possible nonconstant solution to the bounded range $0<\lambda\le\sqrt{-c}$. For $c=0$ this is already impossible. If $c<0$, the summed Riccati equation for $T=F_1$ gives a uniform lower bound for $T$, and the scalar equation then forces $\lambda$ to grow linearly, contradicting the preceding bound. Consequently, the differential system admits no nonconstant solution for $\lambda$ in nonpositive curvature. Returning to the original biharmonic equations yields Theorem~\ref{thm:nonpositive}.

For $c=1$, the term $S-n$ need not have a fixed sign. The trace inequality~\eqref{eq:trace-Cauchy} nevertheless identifies a range in which the continuation argument applies. Proposition~\ref{prop:positive-ODE} gives a strict bound for the distinguished principal curvature along every nonconstant analytic solution. Returning to the hypersurface, this yields a weak global bound for $H^2$. J.-H. Chen's estimate extends to the whole non-CMC hypersurface by a connected-component argument. The strong maximum principle then upgrades the bound for $H^2$ to a strict one. Theorem~\ref{cor:one-point} and Corollary~\ref{cor:scalar-point} follow immediately, while Theorem~\ref{thm:upper-classification} is obtained by combining the one-point CMC criterion with the gap theorem in~\cite{BO2012}.

The principal-curvature reduction developed in~\cite{FuHong2018,FHZ2021,FHZ2023} is combined here with a continuation analysis of the autonomous ODE system arising from the biharmonic equations. This gives a dimension-independent treatment of the Euclidean and hyperbolic cases and, in the spherical case, quantitative rigidity estimates for every possible non-CMC solution.\\

The paper is organized as follows. In Section 2, we recall the fundamental equations for biharmonic hypersurfaces in space forms. Section 3 is devoted to the continuation estimates for the resulting ordinary differential system. Using positive integrating factors and the trace relation, we obtain bounds for all components of the system under the growth conditions required in the proof. In Section 4, we apply these estimates to exclude nonconstant solutions for the distinguished principal curvature in the Euclidean and hyperbolic cases, and complete the proof of Theorem~\ref{thm:nonpositive}. Section 5 deals with biharmonic hypersurfaces in the unit sphere. We establish quantitative estimates for possible non-CMC solutions, derive pointwise criteria for constant mean curvature, and study the associated rigidity and equality cases.

\section{Preliminaries and the principal curvature equations}\label{sec:geometry}

Let $\phi:M^n\to N^{n+1}(c)$ be an isometric immersion. Denote by $\nabla$ and $\widetilde\nabla$ the Levi-Civita connections of $M^n$ and $N^{n+1}(c)$, respectively. For a local unit normal vector field $\xi$, the Gauss and Weingarten formulas are stated as
\begin{align*}
 \widetilde\nabla_XY&=\nabla_XY+h(X,Y),\\
 \widetilde\nabla_X\xi&=-AX,
\end{align*}
where $h(X,Y)=\ip{AX}{Y}\xi$. Thus $\HHvec=H\xi$ and $nH=\trace A$. With the convention
\[
 R(X,Y)Z=\nabla_X\nabla_YZ-\nabla_Y\nabla_XZ-\nabla_{[X,Y]}Z,
\]
the Gauss and Codazzi equations take the form
\[
\begin{aligned}
 R(X,Y)Z
 &=c\bigl(\ip{Y}{Z}X-\ip{X}{Z}Y\bigr)\\
 &\quad+\ip{AY}{Z}AX-\ip{AX}{Z}AY,\\
 (\nabla_XA)Y&=(\nabla_YA)X.
\end{aligned}
\]
We use the Laplacian $\Delta=-\diver\nabla$ on functions defined on $M^n$. With respect to a local orthonormal frame, its action on $f$ is given by
\[
 \Delta f=-\sum_{i=1}^n\bigl(e_ie_i-\nabla_{e_i}e_i\bigr)f.
\]

The normal and tangential parts of the biharmonic equation yield the following standard characterization (see~\cite{CMO2002,OuChen2020,FHZ2023}).

\begin{proposition}\label{prop:biharmonic}
The hypersurface $\phi:M^n\to N^{n+1}(c)$ is biharmonic if and only if
\begin{align}
 &\Delta H+HS=ncH,\label{eq:biharmonic-normal}\\
 &2A\nabla H+nH\nabla H=0,\label{eq:biharmonic-tangent}
\end{align}
where $S=\trace A^2$.
\end{proposition}

We define the open set $U=\{p\in M:\nabla H(p)\ne0\}$, and let $M_A$ be the open dense subset of $M^n$ on which the number of distinct principal curvatures is locally constant. The definition of $U$ is independent of the choice of local unit normal vector field. The principal curvatures and their eigendistributions are smooth on $M_A$.

Suppose that $H$ is nonconstant. Then $U$ is nonempty. All principal-frame calculations below are carried out locally on $U\cap M_A$; throughout the paper, however, the notation $U$ always denotes the full set $\{p\in M:\nabla H(p)\ne0\}$. Equation~\eqref{eq:biharmonic-tangent} shows that $\nabla H$ is a principal direction with principal curvature $-nH/2$. Around each point of $U\cap M_A$, choose an orthonormal principal frame $\{e_1,\ldots,e_n\}$ with $e_n$ parallel to $\nabla H$, and write $Ae_i=\lambda_i e_i$ and $\lambda=\lambda_n=-nH/2$. Then $e_i(\lambda)=0$ for $1\leq i<n$, whereas $e_n(\lambda)\ne0$.

Define the connection coefficients by $\nabla_{e_i}e_j=\sum_k\omega^k_{ij}e_k$. Put $a_i=\omega^n_{ii}$ for $1\le i\le n-1$. For a smooth function $f$, a prime denotes $e_n(f)$. Let $\gamma:I\to U\cap M_A$ be an integral curve of $e_n$, parametrized by $t$, so that
\[
\gamma'(t)=e_n|_{\gamma(t)}.
\]
For simplicity, we use the same notation for the restrictions of $\lambda$, $\lambda_i$ and $a_i$ to $\gamma$. Then
\[
\lambda'=\frac{\mathrm{d}\lambda}{\mathrm{d}t}=e_n(\lambda),\qquad
\lambda_i'=\frac{\mathrm{d}\lambda_i}{\mathrm{d}t}=e_n(\lambda_i),\qquad
a_i'=\frac{\mathrm{d}a_i}{\mathrm{d}t}=e_n(a_i).
\]

 We use the power sums introduced in~\cite{FHZ2023}:
\[
\begin{aligned}
 F_k=\sum_{i=1}^{n-1}a_i^k,\quad H_k=\sum_{i=1}^{n-1}\lambda_i^k,\quad
 G_{l,m}=\sum_{i=1}^{n-1}\lambda_i^l a_i^m.
\end{aligned}
\]
We set $T=F_1$, $F_0=H_0=n-1$, $G_{0,m}=F_m$, and $G_{l,0}=H_l$. The notation $H_k$ denotes a power sum of the transverse principal curvatures.

\begin{lemma}[{\cite{FHZ2023}}]\label{lem:frame}
The principal curvature $\lambda$ is simple on $U\cap M_A$, that is, $\lambda_i\ne\lambda$ for $1\le i\le n-1$.
\end{lemma}

The Gauss and Codazzi equations, together with~\eqref{eq:biharmonic-normal}--\eqref{eq:biharmonic-tangent}, give the following system. We use the formulation developed in \cite[Proposition~3.3]{FHZ2023}.

\begin{proposition}[{\cite{FHZ2023}}]\label{prop:system}
Locally on $U\cap M_A$, the principal curvatures and connection coefficients satisfy
\begin{equation}\label{eq:system}
	\left\{
	\begin{aligned}
		\lambda''
		&=T\lambda'+\lambda(S-nc),\\
		\lambda_i'
		&=(\lambda_i-\lambda)a_i,
		\qquad\quad 1\le i\le n-1,\\
		a_i'
		&=a_i^2+\lambda\lambda_i+c,
		\qquad 1\le i\le n-1,
	\end{aligned}
	\right.
\end{equation}
where
\begin{equation}\label{eq:trace-def}
 \lambda=-\frac13\sum_{i=1}^{n-1}\lambda_i,\qquad
 T=\sum_{i=1}^{n-1}a_i,\qquad
 S=\lambda^2+\sum_{i=1}^{n-1}\lambda_i^2.
\end{equation}
\end{proposition}

The geometric reduction is used locally on $U\cap M_A$. Once $H$ is constant, we return to the normal equation~\eqref{eq:biharmonic-normal}, which gives $H(S-nc)=0$.

The power-sum relations needed below were derived in~\cite[Lemmas~4.1--4.2]{FHZ2023}:
\begin{align}
 T'&=F_2-3\lambda^2+(n-1)c,\label{eq:T-derivative}\\
 G_{1,1}&=\lambda T-3\lambda',\label{eq:G11}\\
 G_{1,1}'&=2G_{1,2}-\lambda F_2+\lambda H_2+cH_1,
 \label{eq:G-recursion}
\end{align}
where $H_1=-3\lambda$ and $H_2=S-\lambda^2$.

We also use the centered Cauchy--Schwarz inequality
\begin{align}
 &\bigl((n-1)G_{1,1}-F_1H_1\bigr)^2\label{eq:Gram-centered}\\
 &\le\bigl((n-1)F_2-F_1^2\bigr)
          \bigl((n-1)H_2-H_1^2\bigr).\notag
\end{align}
Since $(n-1)H_2-H_1^2=\sum_{i<j}(\lambda_i-\lambda_j)^2$, the last factor measures the differences among the transverse principal curvatures. Substituting $H_1=-3\lambda$, $H_2=S-\lambda^2$, and~\eqref{eq:G11} into~\eqref{eq:Gram-centered}, we obtain
\begin{align}
 &\bigl((n+2)\lambda T-3(n-1)\lambda'\bigr)^2\label{eq:moment-defect}\\
 &\le\bigl((n-1)F_2-T^2\bigr)
          \bigl((n-1)S-(n+8)\lambda^2\bigr).\notag
\end{align}
In particular, the trace relation gives
\begin{equation}\label{eq:trace-Cauchy}
\begin{aligned}
 S&\ge\lambda^2+\frac1{n-1}
       \left(\sum_{i=1}^{n-1}\lambda_i\right)^2
  =\frac{n+8}{n-1}\lambda^2.
\end{aligned}
\end{equation}
This estimate holds at every point where $\nabla H\ne0$, including points outside $M_A$, because it uses only the distinguished eigenvalue and the trace of $A$. In the spherical case, this is J.-H. Chen's estimate~\cite{ChenJH1993}; see also~\cite[Lemma~2.1]{BibiLoubeauOniciuc2021}.

The first equation in~\eqref{eq:system} has an equivalent polynomial form. Define
\begin{equation}\label{eq:P-definition}
\begin{aligned}
 \mathcal P={}&3G_{1,2}-3\lambda F_2+\lambda T^2-TG_{1,1}\\
 &+6\lambda H_2+9\lambda^3-3(2n+1)c\lambda.
\end{aligned}
\end{equation}
Since $\lambda=-H_1/3$, the expression $\mathcal P$ is a polynomial in $\lambda_i,a_i$.

\begin{proposition}\label{prop:compatibility}
Every solution of the last two equations of~\eqref{eq:system}, with~\eqref{eq:trace-def}, satisfies
\begin{equation}\label{eq:residual-polynomial}
 \lambda''-T\lambda'-\lambda(S-nc)=-\frac29\mathcal P.
\end{equation}
Consequently, the first equation in~\eqref{eq:system} is equivalent to $\mathcal P=0$ along the solution.
\end{proposition}

\begin{proof}
Differentiating $-3\lambda'=G_{1,1}-\lambda T$, by~\eqref{eq:T-derivative} and~\eqref{eq:G-recursion}, we derive
\[
\begin{aligned}
 -3\lambda''
 &=G_{1,1}'-\lambda'T-\lambda T'\\
 &=2G_{1,2}-\lambda F_2+\lambda H_2+cH_1-\lambda'T\\
 &\quad-\lambda\bigl(F_2-3\lambda^2+(n-1)c\bigr)\\
 &=2G_{1,2}-2\lambda F_2-\lambda'T
   +\lambda H_2+3\lambda^3-(n+2)c\lambda.
\end{aligned}
\]
Using $S=H_2+\lambda^2$ and then $3\lambda'=\lambda T-G_{1,1}$, we obtain
\[
\begin{aligned}
 &\lambda''-T\lambda'-\lambda(S-nc)\\
 &=-\frac23G_{1,2}+\frac23\lambda F_2-\frac23T\lambda'
       -\frac43\lambda H_2-2\lambda^3+\frac{4n+2}{3}c\lambda\\
 &=-\frac29\bigl(3G_{1,2}-3\lambda F_2+\lambda T^2-TG_{1,1}\\
 &\quad+6\lambda H_2+9\lambda^3-3(2n+1)c\lambda\bigr).
\end{aligned}
\]
The last expression is $-2\mathcal P/9$, as required.
\end{proof}

\section{Continuation estimates}\label{sec:estimates}

In this section, we regard~\eqref{eq:system}--\eqref{eq:trace-def} as a system of ordinary differential equations. Throughout this section, we suppose $c\le0$ or $c=1$.

\begin{lemma}\label{lem:analytic}
Let a smooth real solution of~\eqref{eq:system}--\eqref{eq:trace-def} be defined on a nonempty open interval $I_0$ containing $0$. After substituting the trace relation, extend the last two equations of~\eqref{eq:system} to their maximal interval $I\subset\R$, with $I_0\subset I$. Then the first equation of~\eqref{eq:system} holds on $I$.
\end{lemma}

\begin{proof}
Substituting $\lambda=-\frac13\sum_j\lambda_j$ into \eqref{eq:system} gives the autonomous system
\begin{equation}\label{eq:polynomial}
	\left\{
\begin{aligned}
 \lambda_i'&=\left(\lambda_i+\frac13\sum_{j=1}^{n-1}\lambda_j\right)a_i,\\
 a_i'&=a_i^2-\frac13\lambda_i\sum_{j=1}^{n-1}\lambda_j+c,
 \qquad 1\le i\le n-1,
\end{aligned}
	\right.
\end{equation}
where the right-hand sides of this system are polynomials in the $2n-2$ variables $\lambda_i,a_i$. The existence and uniqueness theorem gives a unique maximal solution on an open interval $I$ containing the original interval $I_0$. Since the vector field is polynomial, this solution is real analytic (see~\cite[Sections~2.6 and~4.1]{Teschl2012}).

Define $\lambda,T,S$ on $I$ by~\eqref{eq:trace-def}. These functions are analytic, as are all their derivatives. Consequently, the residual $\lambda''-T\lambda'-\lambda(S-nc)$ is real analytic on $I$. It vanishes on the original nonempty open interval $I_0$. The identity theorem for real analytic functions implies that it vanishes throughout the connected interval $I$, which yields the assertion.
\end{proof}

This continuation is applied to the polynomial system. It does not require an extension of the immersion or its principal frame. The scalar equation is continued from an open interval, rather than from its value at one point.

For each $i$, let us introduce a function $\rho_i$ defined by
  \[
  \rho_i(t)=\rho_i(0)\exp\bigl(-\int_0^t a_i(\tau)\,d\tau\bigr),
  \]
where $\rho_i(0)=\bigl(1+a_i(0)^2+\lambda_i(0)^2+|c|\bigr)^{-1/2}$.
Then
$\rho_i$ is the solution of the differential equation
\begin{align}\label{FU1}
 \rho_i'&=-a_i\rho_i.
\end{align}
$\rho_i$ is well defined and strictly positive on $I$.

\Needspace{18\baselineskip}
\begin{lemma}\label{lem:integrating}
The functions $\rho_i$ satisfy
\begin{align}
 &\bigl[(a_i^2+\lambda_i^2+c)\rho_i^2\bigr]'=0,
 \label{eq:pair-invariant}\\
 &\bigl[(\lambda-\lambda_i)\rho_i\bigr]'=\lambda'\rho_i,
 \label{eq:product-monotonicity}\\
 &\left(\lambda'\prod_{i=1}^{n-1}\rho_i\right)'
 =\lambda(S-nc)\prod_{i=1}^{n-1}\rho_i.
 \label{eq:weighted}
\end{align}
Moreover, $(a_i^2+\lambda_i^2+c)\rho_i^2\le1$ on $I$. In particular, for $c\le0$, one has
\begin{equation}\label{eq:D-nonpositive-bound}
\begin{aligned}
 &|a_i|\rho_i=|\rho_i'|\le1+\sqrt{-c}\,\rho_i,\\
 &|\lambda_i|\rho_i\le1+\sqrt{-c}\,\rho_i
\end{aligned}
\end{equation}
and for $c=1$, one has
\begin{equation}\label{eq:D-positive-bound}
 0<\rho_i\le1,\qquad |a_i|\rho_i\le1,\qquad
 |\lambda_i|\rho_i\le1.
\end{equation}
\end{lemma}

\begin{proof}
Using the last two equations of~\eqref{eq:system}, we compute
\[
\begin{aligned}
 (a_i^2+\lambda_i^2+c)'
 &=2a_i(a_i^2+\lambda\lambda_i+c)
   +2\lambda_i(\lambda_i-\lambda)a_i\\
 &=2a_i(a_i^2+\lambda_i^2+c).
\end{aligned}
\]
The terms containing $\lambda\lambda_i$ cancel. Since $\rho_i'=-a_i\rho_i$, it follows that
\[
\begin{aligned}
 \bigl[(a_i^2+\lambda_i^2+c)\rho_i^2\bigr]'
 &=2a_i(a_i^2+\lambda_i^2+c)\rho_i^2\\
 &\quad-2a_i(a_i^2+\lambda_i^2+c)\rho_i^2=0,
\end{aligned}
\]
which yields~\eqref{eq:pair-invariant}. Using \eqref{FU1} and the second equation of~\eqref{eq:system}, we obtain \eqref{eq:product-monotonicity}.
Since the conserved expression has initial value
\[
 \frac{a_i(0)^2+\lambda_i(0)^2+c}
 {1+a_i(0)^2+\lambda_i(0)^2+|c|}\le1,
\]
it follows from \eqref{eq:pair-invariant} that
$(a_i^2+\lambda_i^2+c)\rho_i^2\le 1$ and hence
\begin{equation}\label{eq:pair-bound}
 (a_i^2+\lambda_i^2)\rho_i^2\le1-c\rho_i^2.
\end{equation}
For $c\le0$, the bound $\sqrt{1-c\rho_i^2}\le1+\sqrt{-c}\,\rho_i$ implies~\eqref{eq:D-nonpositive-bound}. This argument uses only an upper bound for the conserved expression, so it remains valid when that expression is zero or negative. For $c=1$, equation~\eqref{eq:pair-bound} gives~\eqref{eq:D-positive-bound}.

Next, the equations for $\lambda_i$ and $\rho_i$ satisfy
\[
\begin{aligned}
 \bigl[(\lambda-\lambda_i)\rho_i\bigr]'
 &=\bigl(\lambda'-(\lambda_i-\lambda)a_i\bigr)\rho_i
   -(\lambda-\lambda_i)a_i\rho_i\\
 &=\lambda'\rho_i.
\end{aligned}
\]
Finally, differentiating the product yields
\[
 \left(\prod_{i=1}^{n-1}\rho_i\right)'
 =-T\prod_{i=1}^{n-1}\rho_i
\]
and hence
\[
\begin{aligned}
 \left(\lambda'\prod_{i=1}^{n-1}\rho_i\right)'
 &=(\lambda''-T\lambda')\prod_{i=1}^{n-1}\rho_i=\lambda(S-nc)\prod_{i=1}^{n-1}\rho_i,
\end{aligned}
\]
which gives~\eqref{eq:weighted}.
\end{proof}

The fixed-index continuation estimate used below is inspired by the monotonicity argument in~\cite{Abraham2026} in the Euclidean case. We adapt it to the curvature-dependent system~\eqref{eq:polynomial} and record the full argument because the negative curvature term requires a different normalization.

\begin{lemma}\label{lem:bounded-factors}
Let $I$ be the maximal interval of a solution of~\eqref{eq:polynomial}, and assume that the scalar equation in~\eqref{eq:system} holds on $I$. Write $t_*=\sup I$. There is no $t_0\in I$ such that, on $[t_0,t_*)$,
\begin{equation}\label{eq:bounded-hypotheses}
\begin{aligned}
 &\lambda>0,\qquad \lambda'>0,\\
 &\lambda'\prod_{i=1}^{n-1}\rho_i\ge\delta>0,\\
 &0<\rho_i\le C_i,\qquad 1\le i\le n-1,
\end{aligned}
\end{equation}
where $\delta,C_1,\ldots,C_{n-1}$ are constants.
\end{lemma}

\begin{proof}
Suppose that~\eqref{eq:bounded-hypotheses} holds. By~\eqref{eq:pair-bound}, we find
\[
 (a_i^2+\lambda_i^2)\rho_i^2\le1+|c|C_i^2.
\]
Choose a constant $R>0$ larger than all $\sqrt{1+|c|C_i^2}$. Then
\begin{equation}\label{eq:R-bound}
 |a_i|\rho_i\le R,\qquad |\lambda_i|\rho_i\le R,
 \qquad 1\le i\le n-1.
\end{equation}
The bound for the product in~\eqref{eq:bounded-hypotheses} also gives
\begin{equation}\label{eq:lambda-positive-lower}
 \lambda'\ge\frac{\delta}{\prod_{i=1}^{n-1}\rho_i}
            \ge\frac{\delta}{\prod_{i=1}^{n-1}C_i}>0.
\end{equation}
We first bound all pairs $(\lambda_j,a_j)$ except one, and then use the trace relation to control the remaining pair.

At every $t\in[t_0,t_*)$, the equation $\sum_i\lambda_i=-3\lambda$ implies that at least one index $i$ satisfies $\lambda_i(t)\le-3\lambda(t)/(n-1)<0$. For that index, it follows that
\begin{equation}\label{eq:index-bound}
\begin{aligned}
 0<(\lambda-\lambda_i)\rho_i
 &=\left(1+\frac{\lambda}{|\lambda_i|}\right)|\lambda_i|\rho_i\\
 &\le\left(1+\frac{n-1}{3}\right)R
 =\frac{n+2}{3}R=:K.
\end{aligned}
\end{equation}
This index need not be independent of $t$. Choose an increasing sequence $t_k\to t_*$. There are only finitely many indices, so one index satisfies~\eqref{eq:index-bound} along an infinite subsequence. Relabel this index as $1$, and retain the notation $t_k$ for the subsequence.

By~\eqref{eq:product-monotonicity}, $(\lambda-\lambda_1)\rho_1$ is strictly increasing. For any $t<t_*$, choose $k$ with $t<t_k$. Then
\[
 (\lambda(t)-\lambda_1(t))\rho_1(t)
 \le(\lambda(t_k)-\lambda_1(t_k))\rho_1(t_k)\le K.
\]
Choose $t_1$ from this subsequence and put
$\varepsilon=(\lambda(t_1)-\lambda_1(t_1))\rho_1(t_1)>0$. Monotonicity now gives
\begin{equation}\label{eq:fixed-index}
 \varepsilon\le(\lambda-\lambda_1)\rho_1\le K,
 \qquad t_1\le t<t_*.
\end{equation}

For $j\ne1$, equation~\eqref{eq:bounded-hypotheses} implies
\[
 \frac1{\rho_j}\le\frac{\lambda'}{\delta}\prod_{k\ne j}\rho_k,
\]
which together with~\eqref{eq:R-bound} and \eqref{eq:product-monotonicity} yields
\begin{equation}\label{eq:log-control}
\begin{aligned}
 |(\log\rho_j)'|=|a_j|
 &\le\frac R{\rho_j}
 \le\frac{R\lambda'}\delta\prod_{k\ne j}\rho_k\\
 &\le\frac R\delta\left(\prod_{k\ne1,j}C_k\right)\lambda'\rho_1=C_{1j}\bigl[(\lambda-\lambda_1)\rho_1\bigr]',
\end{aligned}
\end{equation}
where $C_{1j}=R\delta^{-1}\prod_{k\ne1,j}C_k$. Unless specified otherwise, all these indices range from $1$ to $n-1$. Integrating~\eqref{eq:log-control} and using \eqref{eq:fixed-index}, we obtain
\[
\begin{aligned}
 \left|\log\frac{\rho_j(t)}{\rho_j(t_1)}\right|
 &\le\int_{t_1}^{t}|(\log\rho_j)'(\tau)|\,d\tau\\
 &\le C_{1j}\bigl((\lambda(t)-\lambda_1(t))\rho_1(t)-\varepsilon\bigr)\\
 &\le C_{1j}(K-\varepsilon).
\end{aligned}
\]
In particular, we have
\begin{equation}\label{eq:D-lower-other}
 \rho_j(t)\ge\rho_j(t_1)e^{-C_{1j}(K-\varepsilon)}>0,
 \qquad t_1\le t<t_*.
\end{equation}
The lower bound is independent of $t$. Equations~\eqref{eq:R-bound} and~\eqref{eq:D-lower-other} therefore bound both $|a_j|$ and $|\lambda_j|$ for every $j\ne1$.

Choose $C\ge0$ such that $|\sum_{j\ne1}\lambda_j|\le C$ on $[t_1,t_*)$. By the trace relation and~\eqref{eq:fixed-index},
\begin{equation}\label{eq:last-denominator}
0<\lambda-\lambda_1
 =4\lambda+\sum_{j\ne1}\lambda_j\le4\lambda+C.
\end{equation}
Both denominators in the following calculation are positive. Since $\lambda'>0$, equation~\eqref{eq:product-monotonicity} gives
\[
\begin{aligned}
 \frac{\lambda'}{4\lambda+C}
 \le\frac{\lambda'}{\lambda-\lambda_1}=\frac{\bigl[(\lambda-\lambda_1)\rho_1\bigr]'}
         {(\lambda-\lambda_1)\rho_1}.
\end{aligned}
\]
Integrating from $t_1$ to $t$ and applying~\eqref{eq:fixed-index}, we obtain
\begin{equation}\label{eq:lambda-log-bound}
\begin{aligned}
 \frac14\log\frac{4\lambda(t)+C}{4\lambda(t_1)+C}
 &\le\log\frac{(\lambda(t)-\lambda_1(t))\rho_1(t)}{\varepsilon}\\
 &\le\log\frac K\varepsilon.
\end{aligned}
\end{equation}
Consequently, we have
\[
 4\lambda(t)+C\le\bigl(4\lambda(t_1)+C\bigr)
                         \left(\frac K\varepsilon\right)^4.
\]
This is a uniform upper bound for $\lambda$. Inserting it into~\eqref{eq:last-denominator} and using~\eqref{eq:fixed-index}, we find
\[
\begin{aligned}
 \rho_1(t)
 &=\frac{(\lambda(t)-\lambda_1(t))\rho_1(t)}
         {\lambda(t)-\lambda_1(t)}\\
 &\ge\frac{\varepsilon}{4\lambda(t)+C}\\
 &\ge\frac{\varepsilon}
 {\bigl(4\lambda(t_1)+C\bigr)(K/\varepsilon)^4}>0.
\end{aligned}
\]
Thus $\rho_1$ also has a uniform positive lower bound. Equation~\eqref{eq:R-bound} bounds $|a_1|$ and $|\lambda_1|$, so every component of the polynomial system is bounded on $[t_1,t_*)$.

Suppose first that $t_*<\infty$. The right-hand sides of~\eqref{eq:polynomial} are bounded on the bounded range just obtained. Hence all $\lambda_i'$ and $a_i'$ are bounded on $[t_1,t_*)$. By integration, each $\lambda_i$ and $a_i$ is uniformly Lipschitz on this interval and has a finite limit as $t\uparrow t_*$. Use these limits as initial data at $t_*$. Local existence gives a solution on a neighborhood of $t_*$, and uniqueness identifies it with the original solution on the overlap. This extends the solution beyond $t_*$, contrary to maximality (see also~\cite[Corollary~2.15]{Teschl2012}).

It follows that $t_*=\infty$. However, integrating~\eqref{eq:lambda-positive-lower} gives
\[
 \lambda(t)\ge\lambda(t_1)
 +\frac{\delta}{\prod_{i=1}^{n-1}C_i}(t-t_1),
 \qquad t\ge t_1.
\]
The right-hand side tends to infinity, contradicting~\eqref{eq:lambda-log-bound}, which completes the proof of Lemma \ref{lem:bounded-factors}.
\end{proof}

\section{Biharmonic hypersurfaces in space forms \texorpdfstring{$N^{n+1}(c)$}{N(n+1,c)} with \texorpdfstring{$c\leq0$}{c <= 0}}\label{sec:nonpositive}

The next proposition for the differential system is useful for proving our main theorem.

\begin{proposition}\label{prop:nonpositive-ODE}
Let $c\le0$. Every smooth real solution of~\eqref{eq:system}--\eqref{eq:trace-def} on a connected open interval satisfies $\lambda\equiv0$.
\end{proposition}

\begin{proof}
Suppose that $\lambda$ is nonconstant. Then there exists a point at which $\lambda'\ne0$. The system is invariant under the simultaneous replacement of $\lambda,\lambda_i$ by $-\lambda,-\lambda_i$, and under the parameter reversal
\[
 (\lambda(t),\lambda_i(t),a_i(t))
 \longmapsto(\lambda(-t),\lambda_i(-t),-a_i(-t)).
\]
After translating the parameter and applying these transformations, we may assume
\[
 \lambda(0)>0,\qquad \lambda'(0)>0.
\]
Let us extend~\eqref{eq:polynomial} to its maximal interval $I$ and write $t_*=\sup I$. By Lemma~\ref{lem:analytic}, the scalar equation remains valid on $I$.

As long as $\lambda>0$, we have $S-nc\ge\lambda^2>0$. Equation~\eqref{eq:weighted} therefore shows that
\[
 \lambda'(t)\prod_{i=1}^{n-1}\rho_i(t)
 \ge \lambda'(0)\prod_{i=1}^{n-1}\rho_i(0)=:\delta>0.
\]
Since each $\rho_i$ is positive, $\lambda'>0$ as long as $\lambda>0$. Hence $\lambda$ cannot return to zero, and on the whole maximal forward interval,
\begin{equation}\label{eq:nonpositive-growth}
 \lambda(t)\ge\lambda(0)>0,\qquad
 \lambda'(t)>0,\qquad
 \lambda'(t)\prod_{i=1}^{n-1}\rho_i(t)\ge\delta>0.
\end{equation}

We first exclude finite-time breakdown. From~\eqref{eq:D-nonpositive-bound},
\[
 \rho_i'\le1+\sqrt{-c}\,\rho_i.
\]
Multiplying by $e^{-\sqrt{-c}\,t}$ and integrating gives
\[
 \rho_i(t)\le e^{\sqrt{-c}\,t}\bigl(\rho_i(0)+t\bigr),
 \qquad 0\le t<t_*.
\]
Thus, if $t_*<\infty$, every $\rho_i$ is bounded above on $[0,t_*)$. Together with~\eqref{eq:nonpositive-growth}, this contradicts Lemma~\ref{lem:bounded-factors}. Consequently,
\[
 t_*=\infty.
\]

We next prove that
\begin{equation}\label{eq:lambda-hyp-bound}
 0<\lambda(0)\le\lambda(t)\le\sqrt{-c},\qquad t\ge0.
\end{equation}
Suppose instead that $\lambda(t_0)>\sqrt{-c}$ at some $t_0\ge0$. Since $\lambda'>0$, the same inequality holds for all $t\ge t_0$. By~\eqref{eq:D-nonpositive-bound},
\[
 (\lambda-\sqrt{-c})\rho_i
 \le(\lambda-\lambda_i)\rho_i+1.
\]
Using~\eqref{eq:product-monotonicity}, we obtain
\[
\begin{aligned}
 \left(
 \frac{(\lambda-\lambda_i)\rho_i+1}{\lambda-\sqrt{-c}}
 \right)'
 &=\frac{\lambda'}{(\lambda-\sqrt{-c})^2}
   \bigl((\lambda_i-\sqrt{-c})\rho_i-1\bigr)\le0.
\end{aligned}
\]
Therefore, for $t\ge t_0$,
\[
\begin{aligned}
 0<\rho_i(t)
 &\le\frac{(\lambda(t)-\lambda_i(t))\rho_i(t)+1}
          {\lambda(t)-\sqrt{-c}}\\
 &\le\frac{(\lambda(t_0)-\lambda_i(t_0))\rho_i(t_0)+1}
          {\lambda(t_0)-\sqrt{-c}}.
\end{aligned}
\]
All the factors $\rho_i$ are then uniformly bounded above on $[t_0,\infty)$, again contradicting Lemma~\ref{lem:bounded-factors}. This shows~\eqref{eq:lambda-hyp-bound}. When $c=0$, the inequality is impossible, so it remains only to consider $c<0$.

Assume now that $c<0$. Summing the equations for $a_i$ and using $\sum_i\lambda_i=-3\lambda$, we obtain
\[
 T'=F_2-3\lambda^2+(n-1)c.
\]
By the Cauchy--Schwarz inequality, $F_2\ge T^2/(n-1)$, while~\eqref{eq:lambda-hyp-bound} gives $\lambda^2\le-c$. Hence
\begin{equation}\label{eq:T-barrier}
 T'\ge\frac{T^2}{n-1}+(n+2)c.
\end{equation}
Choose
\[
 B>\max\left\{|T(0)|,\sqrt{-(n-1)(n+2)c}\right\}.
\]
We claim that $T(t)>-B$ for all $t\ge0$. Otherwise, if $t_1$ is the first point with $T(t_1)=-B$, then $T'(t_1)\le0$, whereas~\eqref{eq:T-barrier} gives
\[
 T'(t_1)\ge\frac{B^2}{n-1}+(n+2)c>0,
\]
a contradiction.

Using $T+B>0$, $\lambda'>0$, $S\ge0$, and $c<0$, the scalar equation yields
\[
\begin{aligned}
 \lambda''+B\lambda'
 &=(T+B)\lambda'+\lambda(S-nc)\\
 &\ge -nc\,\lambda(0)>0.
\end{aligned}
\]
Multiplying by $e^{Bt}$ and integrating gives
\[
 \lambda'(t)\ge e^{-Bt}\lambda'(0)
 +\frac{-nc\,\lambda(0)}{B}(1-e^{-Bt}).
\]
Thus
\[
 \liminf_{t\to\infty}\lambda'(t)
 \ge\frac{-nc\,\lambda(0)}B>0.
\]
It follows that $\lambda(t)$ grows at least linearly for large $t$, contradicting the upper bound $\lambda(t)\le\sqrt{-c}$ in~\eqref{eq:lambda-hyp-bound}.

Therefore $\lambda$ must be constant. Substitution into the scalar equation gives $\lambda(S-nc)=0$. If $\lambda\ne0$, then $S-nc\ge\lambda^2>0$, which is impossible. Hence $\lambda\equiv0$.
\end{proof}

\begin{proof}[Proof of Theorem~\ref{thm:nonpositive}]
Suppose that $H$ is nonconstant on a connected component of $M^n$. Then $U=\{p\in M:\nabla H(p)\ne0\}$ is a nonempty open subset. Since $M_A$ is dense in $M^n$, one has $U\cap M_A\ne\varnothing$. Choose a point of $U\cap M_A$ and a local principal frame as in Section~\ref{sec:geometry}. Along an integral curve of $e_n$ in this neighborhood, Proposition~\ref{prop:system} gives a solution of~\eqref{eq:system}--\eqref{eq:trace-def} with $\lambda'=-ne_n(H)/2\ne0$. This contradicts Proposition~\ref{prop:nonpositive-ODE}. Thus $H$ is locally constant.

Equation~\eqref{eq:biharmonic-normal} now gives $H(S-nc)=0$. If $H\ne0$ on a local normal neighborhood, then $S=nc\le0$. However, by the Cauchy--Schwarz inequality and using $\trace A=nH$ we get $S\ge nH^2>0$, a contradiction. Therefore $H=0$ on every such neighborhood, and the hypersurface is minimal.
\end{proof}

\section{Biharmonic hypersurfaces in \texorpdfstring{$\SSS^{n+1}$}{S(n+1)}}\label{sec:sphere}

We now take $c=1$. The factors $\rho_i$ are bounded above by $1$, but the term $S-n$ in~\eqref{eq:weighted} may change sign. The trace estimate~\eqref{eq:trace-Cauchy} identifies the range in which the continuation argument applies.

\begin{proposition}\label{prop:positive-ODE}
If a smooth real solution of~\eqref{eq:system}--\eqref{eq:trace-def} with $c=1$ has nonconstant $\lambda$, then
\begin{equation}\label{eq:lambda-sphere-bound}
 \lambda^2<\frac{n(n-1)}{n+8}
\end{equation}
throughout its maximal analytic continuation.
\end{proposition}

\begin{proof}
Apply Lemma~\ref{lem:analytic} and pass to the maximal interval $I$. Put $\lambda_*^2=n(n-1)/(n+8)$, where $\lambda_*>0$. We first prove the weak bound $\lambda^2\le\lambda_*^2$.

Suppose that $|\lambda|>\lambda_*$ at some point. The same inequality holds on a neighborhood. Since $\lambda$ is analytic and nonconstant, $\lambda'$ cannot vanish throughout this neighborhood. We may therefore choose a point at which $|\lambda|>\lambda_*$ and $\lambda'\ne0$. Using the sign change and parameter reversal from Section~\ref{sec:nonpositive}, and translating the parameter, we may assume $\lambda(0)>\lambda_*$ and $\lambda'(0)>0$. Equation~\eqref{eq:trace-Cauchy} gives
\begin{equation}\label{eq:sphere-coercivity}
\begin{aligned}
 S-n&\ge\frac{n+8}{n-1}\lambda^2-n\\
    &=\frac{n+8}{n-1}(\lambda^2-\lambda_*^2).
\end{aligned}
\end{equation}
As long as $\lambda>\lambda_*$, the right-hand side is positive. Equation~\eqref{eq:weighted} consequently gives
\[
 \lambda'(t)\prod_{i=1}^{n-1}\rho_i(t)
 \ge\lambda'(0)\prod_{i=1}^{n-1}\rho_i(0)>0.
\]
Thus $\lambda'>0$, and $\lambda$ cannot decrease to $\lambda_*$. These estimates hold throughout the maximal forward interval. By~\eqref{eq:D-positive-bound}, every $\rho_i$ is bounded above by $1$. All the hypotheses of Lemma~\ref{lem:bounded-factors} are therefore satisfied, a contradiction. This shows $\lambda^2\le\lambda_*^2$ on $I$.

We now exclude equality. After changing the sign if necessary, suppose that $\lambda(t_0)=\lambda_*$ at an interior point $t_0\in I$. The weak bound shows that $t_0$ is a maximum point of $\lambda$. Because $\lambda$ is analytic and nonconstant, the first nonzero term of $\lambda-\lambda_*$ at $t_0$ has even order and a negative coefficient. Hence there are an integer $k\ge1$ and a constant $d>0$ such that
\[
 \lambda(t)=\lambda_*-d(t-t_0)^{2k}
             +O\bigl(|t-t_0|^{2k+1}\bigr).
\]
The function $\lambda$ is positive near $t_0$. Combining the scalar equation with~\eqref{eq:sphere-coercivity}, we obtain
\begin{equation}\label{eq:sphere-contact}
 \lambda''-T\lambda'
 \ge\frac{n+8}{n-1}\lambda(\lambda^2-\lambda_*^2).
\end{equation}
The function $T$ is analytic and bounded near $t_0$. Differentiating the convergent Taylor expansion gives
\[
\begin{aligned}
 \lambda''-T\lambda'
 &=-2k(2k-1)d(t-t_0)^{2k-2}
    +O\bigl(|t-t_0|^{2k-1}\bigr),\\
 \frac{n+8}{n-1}\lambda(\lambda^2-\lambda_*^2)
 &=-2nd(t-t_0)^{2k}
    +O\bigl(|t-t_0|^{2k+1}\bigr).
\end{aligned}
\]
For $t>t_0$ sufficiently close to $t_0$, divide~\eqref{eq:sphere-contact} by $(t-t_0)^{2k-2}$ and let $t\to t_0^{+}$. The resulting inequality is $-2k(2k-1)d\ge0$, a contradiction. Thus the equality case is impossible, and~\eqref{eq:lambda-sphere-bound} follows.
\end{proof}

\Needspace{10\baselineskip}
\begin{theorem}\label{thm:sphere}
Let $M^n$ be a connected biharmonic hypersurface in $\SSS^{n+1}$. If the mean curvature is not constant, then
\begin{equation}\label{eq:spherical-global-estimates}
 H^2<\frac{4(n-1)}{n(n+8)}
 \qquad\text{on }M^n.
\end{equation}
If $n\ge5$, then in addition
\begin{equation}\label{eq:scalar-strict-bound}
 \Scal<\frac{n(n^2+10n-20)}{n+8}
 \qquad\text{on }M^n.
\end{equation}
\end{theorem}

\begin{proof}
Suppose that $H$ is nonconstant and let $U=\{p\in M:\nabla H(p)\ne0\}$. At each point of $U\cap M_A$, Proposition~\ref{prop:system}  can be applied along an integral curve of $e_n$, and from Proposition~\ref{prop:positive-ODE} it follows that
\[
 H^2<\frac{4(n-1)}{n(n+8)}.
\]
Since $M_A$ is dense in $M^n$, the set $U\cap M_A$ is dense in $U$. Hence continuity yields
\begin{equation}\label{eq:weak-bound-U}
 H^2\le\frac{4(n-1)}{n(n+8)}\qquad\text{on }U.
\end{equation}

We next extend this bound to all of $M^n$. Define the open set
\[
 V=\left\{p\in M:H^2(p)>\frac{4(n-1)}{n(n+8)}\right\}.
\]
Suppose that $V\ne\varnothing$, and let $V_0$ be a connected component of $V$. Since $M$ is locally connected, $V_0$ is open in $M$. By~\eqref{eq:weak-bound-U}, $V\cap U=\varnothing$, so $H^2$ is constant on $V_0$. We use the boundary argument for connected components in~\cite{NistorOniciuc2019, BibiLoubeauOniciuc2021}.

If $p_0\in\partial_M V_0$, continuity gives $H^2(p_0)\ge4(n-1)/[n(n+8)]$. If the inequality were strict, there would be a connected neighborhood $B$ of $p_0$ contained in $V$. Since $B\cap V_0\ne\varnothing$, the maximality of the connected component $V_0$ would give $B\subset V_0$, contradicting $p_0\in\partial_M V_0$. Thus $H^2(p_0)=4(n-1)/[n(n+8)]$. On the other hand, $H^2$ is constant and strictly greater than this value on $V_0$, which contradicts continuity at $p_0$. Hence $\partial_M V_0=\varnothing$. Since $M$ is connected, $V_0=M$, contrary to the assumption that $H$ is nonconstant. Therefore
\begin{equation}\label{eq:weak-bound-M}
 H^2\le\frac{4(n-1)}{n(n+8)}\qquad\text{on }M^n.
\end{equation}

We now prove the global consequence of J.-H. Chen's estimate~\cite{ChenJH1993} independently of~\eqref{eq:weak-bound-M}. On $U$, the trace inequality~\eqref{eq:trace-Cauchy} gives
\begin{equation}\label{eq:global-S-estimate}
 S\ge\frac{n^2(n+8)}{4(n-1)}H^2.
\end{equation}
Define the open set
\[
 W=\left\{p\in M:S(p)<\frac{n^2(n+8)}{4(n-1)}H^2(p)\right\}.
\]
Suppose that $W\ne\varnothing$, and let $W_0$ be a connected component of $W$. Since $W\cap U=\varnothing$, $H^2$ is constant on $W_0$. This constant is positive, because $S\ge0$. The normal biharmonic equation therefore gives $S=n$ on $W_0$, and $S-n^2(n+8)H^2/[4(n-1)]$ is a fixed negative constant there. If $p_0\in\partial_M W_0$, continuity implies $p_0\in W$. A connected neighborhood of $p_0$ contained in $W$ and meeting $W_0$ must lie in $W_0$, again contradicting the definition of a boundary point. Thus $\partial_M W_0=\varnothing$. By connectedness, $W_0=M$, contrary to the assumption that $H$ is nonconstant. Hence $W=\varnothing$, and~\eqref{eq:global-S-estimate} holds on all of $M$.

It remains to make~\eqref{eq:weak-bound-M} strict. Set
\[
 u=\frac{4(n-1)}{n(n+8)}-H^2.
\]
Then $u$ is smooth and nonnegative. Since $\Delta=-\diver\nabla$, the normal biharmonic equation gives $\diver(\nabla H)=H(S-n)$. Hence
\begin{equation}\label{eq:strong-minimum}
\begin{aligned}
 \diver(\nabla u)
 &=-2H\diver(\nabla H)-2|\nabla H|^2\\
 &=-2H^2(S-n)-2|\nabla H|^2\\
 &\le-2H^2\left(\frac{n^2(n+8)}{4(n-1)}H^2-n\right)\\
 &=\frac{n^2(n+8)}{2(n-1)}H^2u.
\end{aligned}
\end{equation}
If $u(p)=0$ at some point, then $v=-u\le0$ attains its maximum value $0$ at $p$ and satisfies
\[
 \left(\diver\nabla-\frac{n^2(n+8)}{2(n-1)}H^2\right)v\ge0.
\]
The strong maximum principle implies $u\equiv0$, so $H^2$ is constant, a contradiction. Thus the bound~\eqref{eq:spherical-global-estimates} is strict.

Finally, assume $n\ge5$. It follows from the Gauss equation and~\eqref{eq:global-S-estimate} that
\[
\begin{aligned}
 \Scal
 &=n(n-1)+n^2H^2-S\\
 &\le n(n-1)+\frac{3n^2(n-4)}{4(n-1)}H^2\\
 &<n(n-1)+\frac{3n(n-4)}{n+8}
 =\frac{n(n^2+10n-20)}{n+8},
\end{aligned}
\]
where the strict inequality follows from~\eqref{eq:spherical-global-estimates}. This yields~\eqref{eq:scalar-strict-bound}.
\end{proof}

\begin{proof}[Proof of Theorem~\ref{cor:one-point}]
If $H$ is nonconstant, then Theorem~\ref{thm:sphere} would give
\[
 H^2<\frac{4(n-1)}{n(n+8)}
\]
at every point of $M$, contradicting the hypothesis. Hence $H$ is constant.
\end{proof}

\begin{proof}[Proof of Corollary~\ref{cor:scalar-point}]
If $H$ is nonconstant, Theorem~\ref{thm:sphere} would give the strict bound~\eqref{eq:scalar-strict-bound} at every point of $M$, contradicting the hypothesis. Hence $H$ is constant.
\end{proof}

We recall the constant mean curvature gap in~\cite{BO2012}. A connected proper biharmonic CMC hypersurface in $\SSS^{n+1}$ satisfies $S=n$ and $0<H^2\le1$. Moreover, for $n\ge3$,
\begin{equation}\label{eq:CMC-gap}
 0<|H|\le\frac{n-2}{n}\quad\text{or}\quad |H|=1.
\end{equation}
Equality $H^2=1$ characterizes the small hypersphere $\SSS^n(1/\sqrt2)$, while $|H|=(n-2)/n$ holds precisely for the local product~\eqref{eq:Clifford-one}.

\begin{proof}[Proof of Theorem~\ref{thm:upper-classification}]
For $n\ge4$, we have
\[
 \frac{(n-2)^2}{n^2}-\frac{4(n-1)}{n(n+8)}
 =\frac{(n-4)(n^2+4n-8)}{n^2(n+8)}\ge0.
\]
Under either hypothesis, we deduce from Theorem~\ref{cor:one-point} that the mean curvature is  nonzero constant. The gap~\eqref{eq:CMC-gap} and its equality case give the two conclusions.
\end{proof}

The mean curvature threshold in Theorem~\ref{thm:sphere} differs from the endpoint of the classical gap~\eqref{eq:CMC-gap}. The following result describes the consequences of equality at the former threshold.

\begin{proposition}\label{prop:threshold-equality}
Let $M^n$ be a biharmonic hypersurface in $\SSS^{n+1}$. If $H^2(p)=4(n-1)/[n(n+8)]$ at one point, then
\begin{equation}\label{eq:threshold-constant}
 H^2\equiv\frac{4(n-1)}{n(n+8)},\qquad S\equiv n.
\end{equation}
For $n=4$, this equality holds if and only if $M$ is locally congruent to an open part of $\SSS^1(1/\sqrt2)\times\SSS^3(1/\sqrt2)\subset\SSS^5$. For $n\ge5$, equality in the scalar-curvature threshold, $\Scal(p)=n(n^2+10n-20)/(n+8)$, is equivalent to~\eqref{eq:threshold-constant}.
\end{proposition}

\begin{proof}
Note that Theorem~\ref{cor:one-point} gives constant nonzero $H$, and the normal equation gives $S=n$.

For $n=4$, the threshold is $1/4=((n-2)/n)^2$, and the equality case of~\eqref{eq:CMC-gap} gives the Clifford torus product. Conversely, that product has principal curvatures $-1,1,1,1$ up to simultaneous change of sign, and hence $H^2=1/4$.

For $n\ge5$, equality in the scalar-curvature threshold gives constant mean curvature by Corollary~\ref{cor:scalar-point}. The mean curvature cannot vanish, since for a minimal hypersurface $\Scal=n(n-1)-S\le n(n-1)$, strictly below the threshold. Thus $S=n$, and the Gauss equation gives
\[
\begin{aligned}
 n(n-2)+n^2H^2
 &=\frac{n(n^2+10n-20)}{n+8},\\
 H^2&=\frac{4(n-1)}{n(n+8)}.
\end{aligned}
\]
The converse follows from the same equation.
\end{proof}

For $n\ge5$, Proposition~\ref{prop:threshold-equality} reduces equality at the mean-curvature threshold to the CMC system~\eqref{eq:threshold-constant}; no further local classification follows from the present argument. The equality condition for J.-H. Chen's estimate~\eqref{eq:global-S-estimate} is different. At a point where $\nabla H\ne0$, equality in the trace estimate holds exactly when
\begin{equation}\label{eq:norm-equality-spectrum}
\begin{aligned}
 \lambda_n&=-\frac n2H,\\
 \lambda_1=\cdots=\lambda_{n-1}&=\frac{3n}{2(n-1)}H.
\end{aligned}
\end{equation}
The following proposition describes the equality set when the mean curvature is nonconstant.

\begin{proposition}\label{prop:norm-equality}
Let $M^n$ be a connected biharmonic hypersurface in $\SSS^{n+1}$ with nonconstant mean curvature. Then
\[
 S>\frac{n^2(n+8)}{4(n-1)}H^2
\]
holds on $U\cap M_A$. Moreover, the equality set of~\eqref{eq:global-S-estimate} has empty interior. Consequently, the strict inequality holds on an open dense subset of $M^n$.
\end{proposition}

\begin{proof}
Suppose that equality holds at a point of $U\cap M_A$. If $H=0$, then equality gives $S=0$, so every principal curvature vanishes, contradicting the fact that the distinguished eigenvalue $\lambda_n=-nH/2$ is simple on $U\cap M_A$. Thus $H\ne0$. Equality in~\eqref{eq:trace-Cauchy} yields~\eqref{eq:norm-equality-spectrum}. Since the number of distinct principal curvatures is locally constant on $M_A$, the hypersurface has exactly two distinct principal curvatures on a neighborhood of the point. The constant-mean-curvature theorem for this case~\cite{BMO2008} then contradicts $\nabla H\ne0$. This gives the first assertion.

Suppose next that equality holds on a nonempty open set $V$. If $\nabla H\ne0$ somewhere in $V$, then $V\cap U$ contains a nonempty open subset. Since $M_A$ is dense, this subset meets $M_A$, contradicting the first assertion. Hence $H$ is locally constant on $V$. On a connected local normal neighborhood contained in $V$, if the constant value of $H$ is nonzero, the normal equation gives $S=n$, and equality yields
\[
 H^2=\frac{4(n-1)}{n(n+8)},
\]
contrary to Theorem~\ref{thm:sphere}, because $M$ is assumed to have nonconstant mean curvature. If $H=0$ on that neighborhood, the immersion is harmonic there; the unique continuation theorem of Branding and Oniciuc~\cite[Theorem~1.3]{BrandingOniciuc2019} then makes it harmonic on all of $M^n$, again a contradiction. Thus the equality set has empty interior. Since it is closed by continuity, its complement is open and dense.
\end{proof}

In particular, if $S=n^2(n+8)H^2/[4(n-1)]$ holds identically on a connected biharmonic hypersurface, then either $H=0$ and $S=0$, so the hypersurface is totally geodesic, or~\eqref{eq:threshold-constant} holds. No local classification is asserted at an isolated critical point of $H$, or at a point where the principal-curvature multiplicities change.

A constant value of $S$ gives a further application of the continuation estimate.

\begin{proposition}\label{prop:constant-S}
Let $M^n$ be a connected biharmonic hypersurface in $\SSS^{n+1}$. If $S$ is constant and $S\ge n$, then $M$ has constant mean curvature. Moreover, either $M$ is minimal or $S=n$.
\end{proposition}

\begin{proof}
Suppose that $H$ is nonconstant. Choose an integral curve of a local principal field $e_n$ through a point of $U\cap M_A$, and apply Lemma~\ref{lem:analytic}. Let $S_0\ge n$ be the prescribed constant. The function $S-S_0$ is analytic along the extended polynomial solution and vanishes on the original interval. It therefore vanishes throughout the maximal interval $I$.

Let us apply the sign change and parameter reversal used in Section~\ref{sec:nonpositive} to arrange $\lambda(0)>0$ and $\lambda'(0)>0$. As long as $\lambda>0$, equation~\eqref{eq:weighted} gives
\[
 \left(\lambda'\prod_{i=1}^{n-1}\rho_i\right)'
 =\lambda(S_0-n)\prod_{i=1}^{n-1}\rho_i\ge0.
\]
Hence $\lambda'\prod_i\rho_i\ge\lambda'(0)\prod_i\rho_i(0)>0$. It follows that $\lambda'>0$ and $\lambda\ge\lambda(0)>0$ on the entire maximal forward interval. Since $\rho_i\le1$, Lemma~\ref{lem:bounded-factors} gives a contradiction. Thus $H$ is constant. The normal equation then reads $H(S_0-n)=0$, so $S_0=n$ when $M$ is nonminimal.
\end{proof}

The equation $S=S_0$ is essential here and an inequality $S\ge n$ on the initial geometric interval need not persist under analytic continuation.

Combining the one-point criterion with existing gap theorems also gives further complete classification statements. For complete hypersurfaces, Theorem \ref{cor:one-point} can also be combined with Nistor’s gap theorem \cite{Nistor2023}, provided that the mean curvature satisfies the pinching assumptions in that theorem. The one-point criterion supplies the CMC condition whenever the corresponding lower bound is at least $\sqrt{4(n-1)/[n(n+8)]}$.

\begin{remark}
The continuation method leaves open the regime in which $S-n$ changes sign along the analytically continued solution. The estimates above do not rule out this possibility. The classification of all proper CMC biharmonic hypersurfaces in spheres is a separate problem.
\end{remark}

\noindent{\bf Acknowledgement:} {The authors are very grateful to Professor Cezar Oniciuc for his valuable suggestions and comments, especially for pointing out the result of J.-H. Chen. The work was partially supported by NSFC (No.12571051).}\\

\noindent{\bf AI Disclosure:}{
The authors used ChatGPT to assist with language polishing and manuscript preparation. All mathematical arguments and proofs were independently developed  by the authors, who take full responsibility for the content of the manuscript.
}

\end{document}